\documentclass[12pt, reqno]{amsart}
\usepackage[left=1in, right=1in]{geometry} 
\usepackage{hyperref}
\usepackage[utf8]{inputenc}
\usepackage{amssymb}
\usepackage{amscd,amsfonts,amsmath,amssymb, bbm,dsfont, comment}
\usepackage{amsmath}
\usepackage{amsfonts}
\usepackage{amssymb}
\usepackage[utf8]{inputenc}

\usepackage{enumerate,epsf,fancyhdr,float,graphicx,tabularx}
\usepackage{latexsym,mathrsfs,multirow}
\usepackage{wasysym}
\usepackage{xypic}
\usepackage[all]{xy}\usepackage[OT2,T1]{fontenc}
\usepackage[modulo,mathlines,displaymath, running]{lineno}
\usepackage{amsmath,mathrsfs,enumerate,wasysym}
\usepackage{xypic,hhline}
\usepackage[all]{xy}
\usepackage[OT2,T1]{fontenc}
\usepackage[modulo,mathlines,displaymath]{lineno}
\usepackage{tikz,tikz-cd}
\usepackage{dashrule}
\usetikzlibrary{matrix}
\usepackage{ tipa }
\usepackage[modulo,mathlines,displaymath]{lineno}

\newtheorem*{conjecture*}{Conjecture}

\newtheorem{theorem}{Theorem}[section]

\DeclareSymbolFont{cyrletters}{OT2}{wncyr}{m}{n}\DeclareMathSymbol{\Sha}{\mathalpha}{cyrletters}{"58}
\DeclareMathSymbol{\FSha}{\mathalpha}{cyrletters}{"11}
\renewcommand{\phi}{{\varphi}}

\renewcommand{\geq}{\geqslant}
\renewcommand{\leq}{\leqslant}

\newcommand{\cyc}{\mathrm{cyc}}

\newcommand{\ur}{\mathrm{ur}}

\newcommand{\links}{\left(\begin{array}{cc}}
\newcommand{\rechts}{\end{array}\right)}
\newcommand{\bai}{\left[\begin{array}{cc}}
\newcommand{\dai}{\end{array}\right]}
\newcommand{\hidari}{\left(\begin{array}{c}}
\newcommand{\migi}{\end{array}\right)}

\newcommand{\Ff}{\mathbb{F}}

\newcommand{\N}{\mathbb{N}}

\newcommand{\Q}{\mathbb{Q}}

\newcommand{\Z}{\mathbb{Z}}

\newcommand{\calE}{\mathcal{E}}
\newcommand{\calF}{\mathcal{F}}
\newcommand{\calG}{\mathcal{G}}

\newcommand{\calL}{\mathcal{L}}

\newcommand{\calP}{\mathcal{P}}

\newcommand{\calU}{\mathcal{U}}
\newcommand{\calV}{\mathcal{V}}

\newcommand{\gal}{\mathrm{Gal}}

\newcommand{\Gal}{\operatorname{Gal}}

\newcommand{\fla}{\operatorname{fl}}

\newcommand{\pr}{\operatorname{pr}}

\newcommand{\rank}{\operatorname{rank}}
\newcommand{\coker}{\operatorname{coker}}

\newcommand{\Sel}{\operatorname{Sel}}

\newcommand{\injects}{\hookrightarrow}
\newcommand{\lra}{\longrightarrow}

\renewcommand{\contentsname}{Contents\\{\footnotesize\normalfont(A table
of contents should normally not be included)}}

\newcommand{\mystatement}[2]{
    \paragraph{#1}
    \label{#2}
    }

\newcommand{\myref}[2]{\hyperref[#1]{#2}}

\newtheorem{auxiliary proposition}[theorem]{Auxiliary Proposition}

\newtheorem{corollary}[theorem]{Corollary}

\newtheorem{definition}[theorem]{Definition}

\newtheorem{example}[theorem]{Example}

\newtheorem{lemma}[theorem]{Lemma}
\newtheorem{main conjecture}[theorem]{Main Conjecture}
\newtheorem{main theorem}[theorem]{Main Theorem}
\newtheorem{modesty proposition}[theorem]{Modesty Proposition}

\newtheorem{open problem}[theorem]{Open Problem}

\newtheorem{proposition}[theorem]{Proposition}

\newtheorem{remark}[theorem]{Remark}

\newtheorem{convergence lemma}[theorem]{Convergence Lemma}
\newtheorem{corrected lemma}[theorem]{Corrected Lemma}
\newtheorem{growth lemma}[theorem]{Growth Lemma}
\newtheorem{coefficient lemma}[theorem]{Integrality Lemma}
\newtheorem{interpolation lemma}[theorem]{Interpolation Lemma}
\newtheorem{kernel lemma}[theorem]{Kernel Lemma}
\newtheorem{limit lemma}[theorem]{Limit Lemma}
\newtheorem{tandem lemma}[theorem]{Modesty Lemma}
\newtheorem{zero-finding lemma}[theorem]{Zero-Finding Lemma}

\usepackage{hyperref}

\hypersetup{
	colorlinks=true,
	linkcolor=black,
	filecolor=magenta,      
	urlcolor=black,
	pdftitle={Overleaf Example},
	pdfpagemode=FullScreen,
}

\title[~]{ A Study of Fine Selmer Groups Over Function Fields via Greenberg Neighbourhoods}
\author[S. Ghosh]{Sohan Ghosh}
\address[Ghosh]{Indian Institute of Science, Bengaluru, Karnataka (560012) INDIA}
\email{ghoshsohan4@gmail.com}

\keywords{Iwasawa theory; fine Selmer groups; Elliptic Curves; function field}
\subjclass[2020]{Primary: 11R23,  Secondary: 20C07, 11G05, 11R58}
\begin{document}

\maketitle
\begin{abstract}
In \cite{Greenberg73}, Greenberg examined the local behaviour of Iwasawa invariants as functions on the set $\mathcal{E}(F)$, the set of all $\mathbb{Z}_p$-extensions of a number field $F$. Kleine \cite{KleineCanad} later extended these ideas to study the variation of Iwasawa invariants in the context of Selmer groups of abelian varieties across different $\mathbb{Z}_p$-extensions of $F$.

Let $K$ be a global function field of characteristic $p$. In this article, we examine the relationship between the Iwasawa invariants of fine Selmer groups of an elliptic curve over $K$ as one varies over different $\Z_p$-extensions of $K$, using the methods developed by Kleine. We also study generalised Iwasawa invariants of the fine Selmer groups over a $\Z_p^2$-extension $\calL/K$, adapting the techniques introduced in \cite{KleineMatarSujatha}.
\end{abstract}
\section*{Introduction}
\noindent Let us fix a prime $p$ throughout. Iwasawa theory originated in the works of K. Iwasawa, who studied the growth of the $p$-primary parts of class groups in towers of $\mathbb{Z}_p$-extensions. He proved the following celebrated theorem:

\begin{theorem}\cite{IwasawaMain} 
Let $F$ be a number field and $F_\infty$ be a $\mathbb{Z}_p$-extension of $F$. For each $n\geq 0$, let $F_n$ denote the unique subfield of $F_\infty$ containing $F$ such that $\gal(F_n/F) \cong \Z/p^n\Z$. Let $p^{e_n}$ denote the exact power of $p$ that divides $\#Cl(F_n)$. Then there exist invariants $\lambda, \mu \in \mathbb{Z}_{\geq 0}$ and $\nu \in \mathbb{Z}$, depending on $p$ but independent of $n$, such that $e_n = \lambda n + \mu p^n + \nu$ for $n \gg 0$. \end{theorem}

\noindent For the cyclotomic $\mathbb{Z}_p$-extension $F_\cyc$ of $F$, Iwasawa further conjectured that $\mu = 0$. Ferrero and Washington \cite{ferrerowash} later proved this conjecture for abelian number fields; however, it remains open in general.

In 2005, Coates and Sujatha \cite{coatessujatha05} found a deep connection between Iwasawa's $\mu=0$ conjecture and the structure of fine Selmer groups of elliptic curves. More precisely, they proved the following theorem:

\begin{theorem}\cite[Theorem 3.4]{coatessujatha05}\label{thm
} Let $E/F$ be an elliptic curve, and let $p$ be an odd prime such that $F(E_{p^\infty})/F$ is a pro-$p$ extension. Then $R(E/F_\text{cyc})^\vee$ is a finitely generated $\mathbb{Z}_p$-module if and only if Iwasawa’s   $\mu = 0$ conjecture holds for $F_\cyc$. \end{theorem}

\par Motivated by this result, Coates and Sujatha formulated the following conjecture:

\begin{conjecture*}[Conjecture A]\label{conja} 

For all elliptic curves $E$ over $F$, $R(E/F_\text{cyc})^\vee$ is a finitely generated $\mathbb{Z}_p$-module. 

\end{conjecture*}

\par Following   \nameref{conja},  the fine Selmer groups over  cyclotomic $\Z_p$-extensions of a number field have been studied  by various authors  including \cite{LimMurty,jhasujatha, shekhar, wuthrichfine} etc.
A natural question  is to understand how the Iwasawa invariants vary across different $\mathbb{Z}_p$-extensions of $F$. Greenberg approached this problem by introducing a topology on $\mathcal{E}(F)$—the set of all $\mathbb{Z}_p$-extensions of $F$—thereby making it possible to compare Iwasawa invariants on $\mathcal{E}(F)$. Using Fukuda modules, Kleine \cite{Kleine} examined the local behaviour of Iwasawa invariants across $\mathcal{E}(F)$. He later extended this framework to Selmer groups of abelian varieties with good ordinary reduction at primes above $p$ \cite{KleineCanad}. We now briefly recall his result.

 Let $F_\infty/F$ be a $\Z_p$-extension, and let $A$ be an abelian variety defined over  $F$. 
 Assume that $A$ has potentially good ordinary reduction at all primes of $F$ dividing $p$. Suppose further that $A(F_\infty)[p^\infty]$ is finite and that the classical $p$-primary Selmer group $\Sel(A/F_\infty)$ is $\Z_p[[\Gal(F_\infty/F)]]$-cotorsion. Under these hypotheses, Kleine (cf. \cite[Theorem 1.1]{KleineCanad}) proved that there exists a sufficiently small neighbourhood $U$ of $F_\infty$ such that, for every $\Z_p$-extension $\widetilde{F}_\infty$ contained in $U$, the Selmer group $\Sel(A/\widetilde{F}_\infty)$ is also $\Z_p[[\Gal(\widetilde{F}_\infty/F)]]$-cotorsion. Moreover, the $\mu$-invariant of $\Sel(A/\widetilde{F}_\infty)$ is bounded above by that of $\Sel(A/F_\infty)$, and its $\lambda$-invariant is similarly bounded whenever the $\mu$-invariants agree.
Analogous results for fine Selmer groups over number fields have been studied by \cite{KleineMatar}.

In \cite{GhoshJhaShekhar}, the authors initiated the study of fine Selmer groups over function fields of characteristic $p$. These fine Selmer groups are subgroups of the classical Selmer groups. In Section \ref{mainresults}, we investigate the variation of Iwasawa invariants of fine Selmer groups over certain neighbourhoods of $\mathbb{Z}_p$-extensions of a function field $K$ defined over a finite field of characteristic $p$. The main result of that section is Theorem \ref{mainthm:1}. Related results for Selmer groups over function fields in both equal and mixed characteristic were established in \cite{GR}.

We now turn to the asymptotic growth of Iwasawa invariants over $\Z_p^2$-extensions of function fields in characteristic $p$. Our motivation comes from the work of Cuoco, who investigated the asymptotic behaviour of Iwasawa invariants over
 $\Z_p^2$-extensions of a number field. We briefly recall his result here. Let $K$ be  a number field with at least one complex place. Choose two  $\Z_p$-extensions $k_\infty$   and $k_\infty'  $ of $K$ such that $k_\infty \cap k_\infty'=K$. Then $K_\infty=k_\infty  k_\infty'$ is a $\Z_p^2$-extension of $K$ with $G=\gal(K_\infty/K)$.
Let $k_n$ be the unique subfield of $k_\infty$ of degree $p^n$ over $K$. Set $H:=\gal(K_\infty/k_\infty')$ and $H_n:= H^{p^n}\cong p^n\Z_p$. Consider the family of $\Z_p$-extensions $F_{\infty,n}$ of $k_n$, where $F_{\infty,n}=k_nk_\infty'$. Let $Cl(F_{\infty,n})$ denote the  projective limit of ideal class groups of the intermediate number fields of the $\Z_p$-extension $F_{\infty,n}/k_n$. It is a finitely generated  torsion $\Lambda_n:=\Z_p[[\Gal(F_{\infty,n}/ { k_n})]]$-module. In this setting, Cuoco \cite{cuoco} proved that there { exist} constants $l, m_0, m_1, c_1$ and $c_2$ which are independent of $n$, for $n\gg 0$
    \[
    \mu(Cl(F_{\infty,n}))=m_0p^n+m_1n+c_1,\ \text{ and }\lambda(Cl(F_{\infty,n}))=lp^n+c_2.
    \]
Cuoco described the arithmetic significance of these constants $m_0, m_1$ and $l$ in subsequent works. If $Cl(K_\infty)$ is  the projective limit of the
ideal class groups of all intermediate number fields contained in the $\Z_p^2$-extension, then $Cl(K_\infty)$ is a finitely generated $\Z_p[[G]]$-module and one defines generalised Iwasawa invariants (see \S \ref{subsec:generalizedinvariants}). In particular Cuoco \cite{Cuoco82, Cuoco84} (also see \cite[Proposition 6.2]{KleineMatarSujatha}) showed that $m_0=\mu_G(Cl(K_\infty))$ and $m_1=0$ for all but finitely many $\Z_p$-subextensions $K_\infty'$ of $K_\infty$. 

 These results were generalised in \cite{KleineMatarSujatha} for Selmer groups. In Section \ref{sec:asymptoticgrowths}, we adapt the techniques of \cite{KleineMatarSujatha} for fine Selmer groups over function fields of characteristic $p$. Our main result in this section is Theorem \ref{thm:asymgrowth}.

 Our proofs adapt the strategies of \cite{KleineCanad,GR} and \cite{KleineMatarSujatha}; the differences lie in the local analysis at the primes of a finite set $S$ containing the primes of bad reduction of $E$ and the ramified primes, where the local conditions defining the fine Selmer groups $R^S$ are the full local cohomology groups. First, the local kernels in our control theorems involve only the $p$-primary torsion of $E$ over the completions at the primes above $S$, so no assumption on the reduction of $E$ at the ramified primes is needed, whereas \cite[Proposition~3.2]{GR} requires good ordinary reduction there. In particular, Example~\ref{example}, where $E$ has bad reduction at the unique ramified prime, is not covered by \cite{GR}. Second, we retain Greenberg's neighbourhoods rather than the refined neighbourhoods of \cite{KleineCanad,GR}, in which the bad primes that split completely in a given $\mathbb{Z}_p$-extension must also split completely in every $\mathbb{Z}_p$-extension of the neighbourhood; instead, we assume that $E(K_v)[p]=0$ at the primes $v\in S$ that split completely in the given $\mathbb{Z}_p$-extension (here $K_v$ is the completion of $K$ at $v$). Proposition~\ref{prop:split} and Remark~\ref{rem:split} give several cases in which this condition holds, for example, when $E$ is isotrivial (i.e., $j(E)$ is constant) and $p$ is odd. Finally, fine Selmer groups over function fields are much smaller than Selmer groups: if $E$ is non-isotrivial, the fine Selmer groups over the $\mathbb{Z}_p$- and $\mathbb{Z}_p^2$-extensions considered here are killed by a bounded power of $p$ (Remark~\ref{rem:hyptors}). Consequently, the torsion hypotheses hold automatically and all $\lambda$-invariants vanish (Remark~\ref{rem:nonisotrivial}). Thus, the $\lambda$-invariants in our results can be non-zero only when $E$ is isotrivial.

The organisation of the article is as follows. In \S\ref{Preliminaries}, we recall the definitions and preliminary results needed for the article. In \S\ref{mainresults}, we study the variation of Iwasawa invariants of the fine Selmer groups as one ranges over $\Z_p$-extensions of a function field $K$ of characteristic $p$. Finally, in \S\ref{sec:asymptoticgrowths}, we analyse the asymptotic behaviour of generalised Iwasawa invariants associated with fine Selmer groups over a $\Z_p^2$-extension of $K$.

\section{Preliminaries}\label{Preliminaries}
\subsection{Selmer groups and fine Selmer groups}\label{subsec:selfinesel}
 Let $C_K$ be a curve over the finite field $\mathbb{F}$ of characteristic $p$ with function field $K$. Let $E$ be an elliptic curve defined over $K$. Let $U \subset C_K$ be a dense open subset such that $E/K$ has good reduction at every place in $U$. Denote by $\Sigma_K$ the set of all places of $K$, and let $S := \Sigma_K \setminus U$ be the set of places of $K$ outside $U$. For any field $F$ and any prime $v$ of $F$, we write $F_v$ for the completion of $F$ at $v$. Throughout the paper, if $F/K$ is an algebraic extension and $w$ is a
prime of $F$, we write $w\mid S$ if $w$ lies above some prime $v\in S$,
and $w\nmid S$ otherwise. In particular, for a prime $v$ of $K$,
the notations $v\mid S$ and $v\nmid S$ mean $v\in S$ and
$v\notin S$, respectively. For a field $F$, we fix a separable closure $F^{\mathrm{sep}}$ and an algebraic closure $\overline{F}\supseteq F^{\mathrm{sep}}$, and put $G_F:=\operatorname{Gal}(F^{\mathrm{sep}}/F)$.
 
Let $L$ be a finite extension of $K$ contained in $K_S$, the maximal algebraic extension of $K$ unramified outside $S$. Let $v$ be a prime of $K$, and let $w$ denote a prime of $L$. Define

\begin{equation}\label{equationdefine}
	\begin{aligned}
		J_v^1(E/L):=\underset{w\mid v}\prod	\frac{H^1_{\fla}(L_w, E_{p^\infty})}{\operatorname{im}(\kappa_w)} \ \text{and }
		K_v^1(E/L):= \underset{w\mid v}\prod	H^1_{\fla}(L_w, E_{p^\infty}).	
	\end{aligned}
\end{equation}

 Here $H^i_{\fla}(-,-) $ denotes the flat cohomology \cite[Chapters II, III]{Milne1986}, and $\kappa_w:E(L_w)\otimes {\Q_p}/{\Z_p}\injects H^1_{\fla}(L_w,E_{p^\infty})$ is induced by the Kummer map. 
\begin{definition}\cite[Prop. 2.4]{KatoTrihan}\label{defn:selmer}
	With $\Sigma_K, S \text{ and } K\subset L \subset K_S$ as above, we define the Selmer group of $E$ over $L$ as:
	\begin{equation}\label{definition24}
		\mathrm{Sel}(E/L):= \ker\left(H^1_{\fla}(L, {E}_{p^\infty})\lra \underset{v\in \Sigma_K}\prod J_v^1(E/L)\right).
	\end{equation}
\end{definition}
\begin{definition}
   With notations as in Definition \ref{defn:selmer}, we define the $S$-fine Selmer group as:
	\begin{equation}\label{definition4.1}
		\begin{aligned}
			R^S(E/L):={} \ker\left(H^1_{\fla}(L, E_{p^\infty})\lra \underset{v\in S}\prod K_v^1(E/L) \underset{v\in \Sigma_K\setminus S}\prod J_v^1(E/L)\right)\\
			\cong \ker\left (\mathrm{Sel}(E/L)\lra \underset{w\mid v, v\in S}\bigoplus E(L_w)\otimes \Q_p/\Z_p\right).
		\end{aligned}
	\end{equation}
\end{definition}
For an infinite algebraic extension $\mathcal{L}$ of $K$, the definitions of the Selmer group and the fine Selmer group extend in the usual way by taking  inductive limits over the finite extensions of $K$ contained in $\mathcal{L}$.
\subsubsection{Dependence on $S$}

{
One can show that the $S$-fine Selmer group $R^S(E/K_\infty)$ is independent
of the choice of $S$, for a sufficiently large $S$ containing
all primes of bad reduction of $E$, where $K_\infty$ is a $\Z_p$-extension of
$K$ (see \cite[Remark 3.2]{GhoshJhaShekhar}); however, providing an explicit
description of such a set $S$ appears to be difficult. }

Now let $\calL/K$ be a Galois extension containing the unramified $\mathbb{Z}_p$-extension of $K$ (defined in \S\ref{subsec:examplesofZpextensions}), unramified outside a finite set of primes, with $\mathrm{Gal}(\calL/K)$ a pro-$p$ $p$-adic Lie group without elements of order $p$. Assume $S\neq\emptyset$, let $\mathcal{E}$ be the abelian scheme extending $E$ over $U$, and assume that the schematic closure of $E[p]^{\natural}:=E(K^{\mathrm{sep}})[p]$ in $\mathcal{E}[p]$ is étale over $U$. Then $R^{S'}(E/\calL)=R^S(E/\calL)$ for every finite set $S'\supseteq S$ (\cite[Theorem 7.1]{GhoshRaySuzuki}).

\subsection{ Examples of $\Z_p$-extensions of function fields}\label{subsec:examplesofZpextensions}
 Let $\Ff^{(p)}$ be the unique subfield of $\overline{\Ff}$, the algebraic closure of $\Ff$, such that $\gal(\Ff^{(p)}/\Ff)\cong \Z_p$. Set $K_\infty^\mathrm{ur}:=K\Ff^{(p)}$. Note that $K_\infty^\ur/K$ is unramified everywhere.  This $\Z_p$-extension $K_\infty^\ur$ is referred to in the literature as the {\it arithmetic} $\Z_p$-extension. Unlike the case of number fields, for function fields one can construct infinitely many independent $\Z_p^d$-extensions of $K$ for any $d \geq 1$. We now describe  examples of $\Z_p^d$-extensions constructed from Carlitz modules. These extensions are obtained as follows:

 	Let $K=\Ff(t)$ and $P(t)=a_nt^n+\cdots +a_0 \in \Ff [t]$. We define the Carlitz polynomial $[P(t)](X)$ with coefficients in $\Ff[t]$ recursively as follows:

		$[1] (X)=X$, 
		
		$[t] (X)=X^p +tX$,
		
		$[t^n] (X)=[t]([t^{n-1}](X))$ and
		
		$[a_n t^n+\cdots+a_1 t+a_0](X)= a_n[t^n](X)+\cdots +a_1[t](X)+a_0 (X).$
		
		Consider a field extension $F$ of $K$. Then $F$ can be thought of as a $\Ff[t]$-module, where the action of $\Ff[t]$  is given by the Carlitz polynomials. 
		Choose  a prime ideal $\mathfrak P$ of $\Ff[t]$. For $n>0$, let
		$$\Lambda_{\mathfrak P^n}:=\{\lambda\in \overline{\Ff(t)}| [\mathfrak{P}^n](\lambda)=0\}.$$
		
		Then the extension $K(\Lambda_{\mathfrak P ^n})/K$ is  Galois,  with $\gal(K(\Lambda_{\mathfrak P ^n})/K)\cong (\Ff[t]/\mathfrak{P}^n)^\times$. Put $\widetilde{K}:=\underset{n\geq 1}\bigcup K(\Lambda_{\mathfrak{P}^n})$, then $\gal(\widetilde K /K)\cong \Z_p^\N \times (\Ff[t]/\mathfrak{P})^\times$. 

  The $\Z_p^d$-extensions obtained from $\widetilde {K}$, for $d\geq 1$, are ramified only at the prime $\mathfrak P$, and  are totally ramified at that prime \cite[Proposition 12.7]{Rosen}. The resulting $\Z_p$-extensions are examples of  {\it geometric} $\Z_p$-extensions.
\subsection{Greenberg topology}
For a $\Z_p$-extension $K_\infty/K$, we define $\mathcal{P}(K_\infty)$ to be the set of primes of $K$ that ramify in $K_\infty$.
Let $\mathcal{E}(K)$ denote the set of all  finitely ramified $\Z_p$-extensions of $K$. For $n\geq 0$, define
\[
\mathcal{U}(K_\infty,n):=\{K_\infty'\in \mathcal{E} (K)| [K_\infty\cap K_\infty': K]\geq p^n, \mathcal{P}(K_\infty')\subset\mathcal{P}(K_\infty) \}.
\]
We endow $\mathcal{E} (K)$ with the topology generated by these sets. Following \cite{Greenberg73}, this notion of neighbourhood was introduced in \cite{Kleine}.
\iffalse
Let $E$ be an elliptic curve defined over $K$.  Define the set
\[
S_1:=\{v\in\Sigma_K\ |\ v \textrm{ ramifies in } K_\infty\}\bigcup \{v\in \Sigma_K\ |\ E \textrm{ has bad reduction at } v\}.
\]
 Note that $S_1$ is  finite. Let $S$ be a finite set of primes of $K$such that $S_1 \subseteq S$, and assume that for any prime $v \in S \setminus S_1$, the prime $v$ splits completely in $K_\infty$. 

 We now consider a refinement of the above topology. For ${K}_\infty\in \mathcal{E}(K)$ and $m\in \N$, let $\calV( K_\infty,m)\subset \mathcal{U}(K_\infty,m)$ denote the subset of $\Z_p$-extensions $\widetilde{K}_\infty$ of $K$ in $\mathcal{U}(K_\infty,m)$ such that each prime of $S\setminus S_1$ (that is totally
split in $K_\infty$)  splits completely  in $\widetilde{K}_\infty$. We equip $\mathcal{E}(K)$ with the topology generated by the sets  $\calV(K_\infty,m)$, i.e, the smallest topology  on $\calE(K)$ generated by these sets. This is a finer topology on $\mathcal{E}(K)$ than the topology generated by the sets 
 $\{\mathcal{U}(K_\infty,m), m \geq 0\}$. Indeed it is easy to see that if $K_\infty^\prime \in \mathcal{U}(K_\infty,m)$, then $L_\infty \in \calV(E, K_\infty',m+1) \subset \mathcal{U}(K_\infty,m)$. Recall, that a similar type of refined neighbourhoods were also considered in \cite{KleineCanad}.
\fi
\subsection{Fukuda Modules}
 In \cite{KleineCanad}, Kleine introduced a special family of Iwasawa modules $X = \varprojlim X_n$. These modules were designed to derive information about the projective limit by utilising data from a sufficiently large number of layers $X_n$.
\begin{definition}\label{defn:Fukudamodules}
Let $1 \le d \in \mathbb{N}$. Let $R$ be the ring of formal power series over $\mathbb{Z}_p$ in $d$ variables with maximal ideal $\mathfrak{m} = (p, T_1, \ldots, T_d)$. Let $X = \varprojlim X_n$ denote the projective limit of $R$-modules $X_n$, $n \in \mathbb{N}$, each assumed to be an abelian pro-$p$ group. Further assume that $X$ is compact as an $R$-module with respect to the $\mathfrak{m}$-adic topology. Let $Y_n \subset X$ denote the kernel of the projection map $X \to X_n$ for each $n$. Let $C_1, C_2, C_3 \in \mathbb{N}$ be $p$-powers. Then $X$ is called a Fukuda $R$-module with parameters $(C_1, C_2, C_3)$ if there exists a family of compact $R$-submodules $(Z_n)_n$ of $X$ such that 
$$|\coker(\pr_n)| \le C_1, Z_{n+1} \subset \mathfrak{m} \cdot Z_n, [Y_n : (Y_n \cap Z_n)] \le C_2,$$
and
$$[Z_n : (Y_n \cap Z_n)] \le C_3$$ for each $n \in \mathbb{N}$.
\end{definition}

Let $F$ be a number field. Consider a $\mathbb{Z}_p^d$-extension $F_\infty/F$, where $d \geq 1$. Write $F_\infty = \bigcup_{n} F_n$ with $\Gal(F_n/F) \cong (\Z/p^n\Z)^d$. Let $X = \varprojlim X_n$, where $X_n$ is the $p$-primary subgroup of the ideal class group of $F_n$. Then, under appropriate assumptions on the ramification of primes in $F_\infty/F$, one can show that $X$ is a Fukuda-$R$-module with parameters $(1,1,1)$ (see \cite[\S 3]{KleineCanad}), where $R \cong \mathbb{Z}_p[[\Gal(F_\infty/F)]]$.

Before proceeding further, we recall the following results from \cite{BandiniLonghiIJNT}.
\noindent \begin{lemma}\cite[Lemma 4.1]{BandiniLonghiIJNT}\label{lemma:BandiniLonghiIJNT}
    Let $\Gamma \cong \mathbb{Z}_p^d$ and $B$ a finite $p$-primary $\Gamma$-module. Then 
\[
|H^1(\Gamma, B)| \leq |B|^d \quad \text{and} \quad |H^2(\Gamma, B)| \leq |B|^{\frac{d(d-1)}{2}}.
\]
\end{lemma}
\noindent\begin{lemma}\label{lemma:finitenesslocal}
    Let $E/K$ be a non-isotrivial elliptic curve (i.e., $j(E)\notin \Ff$). Let $\calF_d$ be a $\Z_p^d$-extension of $K$, and let $w$ be a prime of $\calF_d$.  Then the group $E[p^\infty](\calF_{d,w})$ is finite.
\end{lemma}
\begin{proof}
    The proof of this lemma follows from \cite[Lemma 4.3, proof of Theorem 4.12 ]{BandiniLonghiIJNT}.
\end{proof}
\subsection{Generalised Iwasawa invariants}\label{subsec:generalizedinvariants}
 Let $\calG$ be a free $\Z_p$-module of rank $d$, written multiplicatively, and let
$\sigma_1,\cdots,\sigma_d$ be a fixed basis of $\calG$. Let $\Lambda(\calG)$ be the completed $\Z_p$ group-ring of $\calG$ and set
$X_j=\sigma_j-1$. Then $\Lambda(\calG)$ is a power series ring $\Z_p[[X_1, \cdots , X_d]]$, and $\calG$ is the closed
multiplicative subgroup of $\Lambda(\calG)$ generated by the $1 +X_j$. Note that each element of $GL_d(\Z_p)$ induces an automorphism of $\calG$ that extends to a ring automorphism of $\Lambda(\calG)$ via $\Z_p$-linearity. We refer to such automorphisms as "linear automorphisms" of $\Lambda(\calG)$ (cf. \cite[\S 1]{CuocoMonsky}).
Suppose $F\neq 0 \in \Lambda(\calG)$. We now define two integers $m_0$ and $l_0$ associated with $F$.
\begin{definition}\cite[Definition 1.1]{CuocoMonsky}\label{defn:constantm0}
    The constant $m_0(F)$ is the largest power of $p$  dividing $F$. 
\end{definition}

Suppose now that $\sigma\in \calG\setminus \calG^p$. Then there exists a linear automorphism of $\mathcal{G}$ that maps $\sigma$ to $1+X_i$, for some $i>0$. 
It follows that $(\overline{\sigma-1})$, the mod-$p$ reduction of $\sigma-1$, is a height $1$ prime ideal in the reduction $\Omega(\calG):=\Lambda(\calG)/p\Lambda(\calG)$. If $P$ is a height $1$ prime ideal of $\Omega(\calG)$, let $\upsilon_P$ be the associated discrete valuation, which assigns to each non-zero element $\bar{F}$ in $\Omega(\calG)$ the highest power of $P$ dividing $\bar{F}$.

\begin{definition}\cite[Definition 1.2]{CuocoMonsky}\label{defn:constantl0}
    Write $F=p^{m_0}F_0$  with $F_0\neq 0$ and $p\nmid F_0$. Then, 
    \[
    l_0(F):=\underset{P}\Sigma\ \upsilon_P (\overline{F_0}),
    \]
 the sum
extending over all $P$ of the form $\overline{\sigma-1},$ where $\sigma\in \calG \setminus \calG^p$. 
\end{definition}
Note that this is a finite sum, since the ring $\Omega(\calG)$ is a Noetherian unique factorisation domain.

A finitely generated torsion $\Lambda(\calG)$-module is called pseudo-null if it is annihilated by two relatively prime elements of $\Lambda(\calG)$. By a general structure theorem, if $M$ is finitely generated and $\Lambda(\calG)$-torsion, then there exists $\Lambda(\calG)$-pseudo-isomorphism
(i.e., $\Lambda(\calG)$-module homomorphisms with pseudo-null kernels and cokernels)  $\psi : M \lra E_M$,
where $E_M:=\bigoplus\limits_{j=1}^s \Lambda(\calG)/(f_j) $ for suitable $f_1, \cdots, f_s\in \Lambda(\calG)$. The module $E_M$ is called an elementary torsion $\Lambda(\calG)$-module
attached to $M$.

\section{Growth of Iwasawa invariants over $\Z_p$-extensions}\label{mainresults}
Let $K$ be a function field over a finite field of characteristic $p$. Let $K_\infty$ be a finitely ramified $\Z_p$-extension of $K$  and $\Gamma:=\Gal(K_\infty/K)$. We exclude the case where $K_\infty=K_\infty^\ur$, the unramified $\Z_p$-extension, since $\calU(K_\infty^\ur,m)$ is a singleton set  for each $m\geq 0$. Let $E$ be an  elliptic curve defined over $K$. Let $S$ be a finite set of primes of $K$ containing all primes that ramify in $K_\infty$ and all primes at which $E$ has bad reduction. We fix this set $S$ throughout the section. Note that if $\widetilde{K}_\infty\in \calU(K_\infty,m)$, then by definition $\calP(\widetilde{K}_\infty)\subset \calP(K_\infty)$. Therefore, we can choose $S$ uniformly in Proposition \ref{prop:FUkudaS1}.
We first prove a control theorem for the natural map $R^S(E/K_n)\lra R^S(E/K_\infty)^{\Gamma_n}$.
\begin{theorem}\label{thm:control}
     Let $E$ be an  elliptic curve  defined over $K$. Let $K_\infty$ be a finitely ramified $\Z_p$-extension of $K$. Furthermore, assume that for every prime $w\mid S$ of $K_\infty$,  $E(K_{\infty,w})[p^\infty]$ is finite.
     Then the kernel and the cokernel of the natural map
    \[
    R^S(E/K_n)\lra R^S(E/K_\infty)^{\Gamma_n}
    \]
    are finite and bounded independently of $n$.
\end{theorem}
\begin{proof}
Consider the following commutative diagram:
% https://q.uiver.app/#q=WzAsOCxbMCwwLCIwIl0sWzEsMCwiUl5TKEUvS19cXGluZnR5KV57XFxHYW1tYV9ufSJdLFsyLDAsIkheMV97XFxmbGF9KEtfXFxpbmZ0eSwgRVtwXlxcaW5mdHldKV57XFxHYW1tYV9ufSJdLFszLDAsIlxcbGVmdChcXHVuZGVyc2V0e3ZcXG1pZCBTfVxcYmlnb3BsdXMgS14xX3YoRS9LX1xcaW5mdHkpXFx1bmRlcnNldHt2XFxubWlkIFN9XFxiaWdvcGx1cyBKXjFfdihFL0tfXFxpbmZ0eSlcXHJpZ2h0KV57XFxHYW1tYV9ufSJdLFswLDEsIjAiXSxbMSwxLCJSXlMoRS9LX24pIl0sWzIsMSwiSF4xX3tcXGZsYX0oS19uLCBFW3BeXFxpbmZ0eV0pIl0sWzMsMSwiXFx1bmRlcnNldHt2XFxtaWQgU31cXGJpZ29wbHVzIEteMV92KEUvS19uKVxcdW5kZXJzZXR7dlxcbm1pZCBTfVxcYmlnb3BsdXMgSl4xX3YoRS9LX24pIl0sWzAsMV0sWzEsMl0sWzIsM10sWzQsNV0sWzUsNl0sWzYsN10sWzUsMSwiXFxhbHBoYV9uIl0sWzYsMiwiXFxiZXRhX24iXSxbNywzLCJcXGdhbW1hX249XFx1bmRlcnNldHt2fVxcb3BsdXMgXFxnYW1tYV97bix2fSJdXQ==
\[\begin{tikzcd}
	0 & {R^S(E/K_\infty)^{\Gamma_n}} & {H^1_{\fla}(K_\infty, E[p^\infty])^{\Gamma_n}} & {\left(\underset{v\mid S}\prod\ K^1_v(E/K_\infty)\ \underset{v\nmid S}\prod\ J^1_v(E/K_\infty)\right)^{\Gamma_n}} \\
	0 & {R^S(E/K_n)} & {H^1_{\fla}(K_n, E[p^\infty])} & {\underset{v\mid S}\prod\ K^1_v(E/K_n)\ \underset{v\nmid S}\prod\ J^1_v(E/K_n)}
	\arrow[from=1-1, to=1-2]
	\arrow[from=1-2, to=1-3]
	\arrow[from=1-3, to=1-4]
	\arrow[from=2-1, to=2-2]
	\arrow["{\delta_n}", from=2-2, to=1-2]
	\arrow[from=2-2, to=2-3]
	\arrow["{\beta_n}", from=2-3, to=1-3]
	\arrow[from=2-3, to=2-4]
	\arrow["{\gamma_n=\underset{v}\prod \gamma_{n,v}}", from=2-4, to=1-4]
\end{tikzcd}\]
First, observe that $\ker(\beta_n) = H^1(\Gamma_n, E(K_\infty)[p^\infty])$, which is finite by our assumption. Moreover, by Lemma \ref{lemma:BandiniLonghiIJNT}, its order is bounded above by $|E(K_\infty)[p^\infty]|$. Since $\Gamma\cong \Z_p$ has $p$-cohomological dimension 1,  $\coker(\beta_n)=0$.

Next, let $v \mid S$ be a prime of $K$, and fix a prime $w \mid v$ of $K_\infty$. Again by our assumption, the group $E(K_{\infty,w})[p^\infty]$ is finite. By Shapiro's lemma, $\ker(\gamma_{n,v})$ is the product, over the primes of $K_n$ above $v$, of the groups $H^1(\Gamma_{n,w},E(K_{\infty,w})[p^\infty])$, where $\Gamma_{n,w}\subseteq\Gamma_n$ is the decomposition group; each has order at most $|E(K_{\infty,w})[p^\infty]|$ by Lemma \ref{lemma:BandiniLonghiIJNT}. If $v$ splits completely in $K_\infty$, then $\ker(\gamma_{n,v})=0$. Otherwise there are at most $g_v$ factors, $g_v$ being the number of primes of $K_\infty$ above $v$, so $|\ker(\gamma_{n,v})|\le|E(K_{\infty,w})[p^\infty]|^{g_v}$. For primes $v \notin S$, we have $\ker(\gamma_{n,v}) = 0$ \cite[Lemma 3.3.1]{Tan}. The result then follows by applying the snake lemma.
\end{proof}

%It follows that, for each $v \mid S$, the kernel $\ker(\gamma_{n,v})$ is bounded by $|E(K_{\infty,w})[p^\infty]|$ by applying Shapiro’s lemma and Lemma \ref{lemma:BandiniLonghiIJNT}.

\begin{remark}
\begin{enumerate}
    \item  If $E/K$ is a non-isotrivial elliptic curve, then the assumption that for every $w\mid S$,  $E(K_{\infty,w})[p^\infty]$ is finite holds by Lemma \ref{lemma:finitenesslocal}.
    \item { By \cite[Proposition~2.11]{TanMathAnn},
$E(K_\infty)[p^\infty]$ is finite unless $K_\infty/K$
is the unramified $\mathbb{Z}_p$-extension and $E$ has no
place of split multiplicative reduction.}

    %By \cite[Proposition 2.11]{TanMathAnn}, the group $E(K_\infty)[p^\infty]$ is finite provided that $K_\infty$ is not the unramified $\Z_p$-extension and that $E$ does not have split multiplicative reduction at any prime of $K$.
\end{enumerate}
  
\end{remark}
For a compact or discrete $\Z_p$-module $M$, let
$ M^\vee := \operatorname{Hom}_{\mathrm{cont}}(M, \Q_p/\Z_p)$ denote its Pontryagin dual. As a corollary, we obtain the following result:

\begin{corollary}\label{cor:fukuda}
Under the assumptions of Theorem~\ref{thm:control}, $R^S(E/K_\infty)^\vee=\varprojlim R^S(E/K_n)^\vee$ is a Fukuda-$\Lambda$-module with parameters $(C_1,C_2,1)$, for suitable $p$-powers $C_1$ and $C_2$.
\end{corollary}
\begin{proof}
We argue as in \cite[Corollary~4.2]{KleineCanad}, using the notation of Definition \ref{defn:Fukudamodules} (with $R=\Lambda$) and of the proof of Theorem~\ref{thm:control}, in particular the map $\delta_n$. Put $X=R^S(E/K_\infty)^\vee$ and $X_n=R^S(E/K_n)^\vee$, fix a topological generator $\gamma$ of $\Gamma$, and set $\omega_n=\gamma^{p^n}-1$. The projection $\mathrm{pr}_n\colon X\to X_n$ is the Pontryagin dual of $\delta_n$ followed by the inclusion $R^S(E/K_\infty)^{\Gamma_n}\subseteq R^S(E/K_\infty)$, so it factors as $X\twoheadrightarrow X/\omega_nX\to X_n$, the second map being $\delta_n^\vee$. Hence
\[
|\operatorname{coker}(\mathrm{pr}_n)|=|\ker(\delta_n)|,\qquad \omega_nX\subseteq Y_n,\qquad [Y_n:\omega_nX]=|\operatorname{coker}(\delta_n)|.
\]
Take $Z_n=\omega_nX$. Then $Z_{n+1}=\nu_nZ_n\subseteq\mathfrak mZ_n$, where $\nu_n=\omega_{n+1}/\omega_n=\sum_{i=0}^{p-1}\gamma^{ip^n}\in\mathfrak m$. Since $Z_n\subseteq Y_n$, we have $[Z_n:Y_n\cap Z_n]=1$ and $[Y_n:Y_n\cap Z_n]=|\operatorname{coker}(\delta_n)|$. The proof of Theorem~\ref{thm:control} gives
\[
|\ker(\delta_n)|\le|\ker(\beta_n)|\le|E(K_\infty)[p^\infty]|,\qquad
|\operatorname{coker}(\delta_n)|\le|\ker(\gamma_n)|\le\prod_{v}|E(K_{\infty,w_v})[p^\infty]|^{g_v},
\]
where $v$ runs over the primes of $S$ that do not split completely in $K_\infty$, $w_v$ is a fixed prime of $K_\infty$ above $v$, and $g_v$ is the number of primes of $K_\infty$ above $v$. Thus one can take
\[
C_1=|E(K_\infty)[p^\infty]|\qquad\text{and}\qquad C_2=\prod_{v}|E(K_{\infty,w_v})[p^\infty]|^{g_v}. \qedhere
\]
\end{proof}

\noindent For the proof of the next proposition, we closely follow the proof of \cite[Theorem 4.5]{KleineCanad}{.}
\begin{proposition}\label{prop:FUkudaS1}
Let us keep the hypotheses of Theorem \ref{thm:control}. {
Assume, in addition, that $E(K_v)[p]=0$ for every prime
$v\in S$ that splits completely in $K_\infty$.} Then there exist an integer $m\geq 1$ and $C_1, C_2 \in \N$ such that $R^{S}(E/\widetilde{K}_\infty)^\vee$ is a Fukuda-$\Lambda$-module with  parameters $(C_1, C_2, 1)$ for each $\widetilde{K}_\infty \in \calU(K_\infty,m)$.
\end{proposition}
\begin{proof}
    For each $n>0$, we have the following commutative diagram:
   % https://q.uiver.app/#q=WzAsOCxbMCwwLCIwIl0sWzEsMCwiUl5TKEUvXFx3aWRldGlsZGV7S31fXFxpbmZ0eSlee1xcR2FtbWFfbn0iXSxbMiwwLCJIXjFfe1xcZmxhfShcXHdpZGV0aWxkZXtLfV9cXGluZnR5LCBFW3BeXFxpbmZ0eV0pXntcXEdhbW1hX259Il0sWzMsMCwiXFxsZWZ0KFxcdW5kZXJzZXR7dlxcbWlkIFN9XFxiaWdvcGx1c1xcIEteMV92KEUvXFx3aWRldGlsZGV7S31fXFxpbmZ0eSlcXCBcXHVuZGVyc2V0e3ZcXG5taWQgU31cXCBcXGJpZ29wbHVzIEpeMV92KEUvXFx3aWRldGlsZGV7S31fXFxpbmZ0eSlcXHJpZ2h0KV57XFxHYW1tYV9ufSJdLFswLDEsIjAiXSxbMSwxLCJSXlMoRS9cXHdpZGV0aWxkZXtLfV9uKSJdLFsyLDEsIkheMV97XFxmbGF9KFxcd2lkZXRpbGRle0t9X24sIEVbcF5cXGluZnR5XSkiXSxbMywxLCJcXHVuZGVyc2V0e3ZcXG1pZCBTfVxcYmlnb3BsdXNcXCBLXjFfdihFL1xcd2lkZXRpbGRle0t9X24pXFwgXFx1bmRlcnNldHt2XFxubWlkIFN9XFxiaWdvcGx1c1xcIEpeMV92KEUvXFx3aWRldGlsZGV7S31fbikiXSxbMCwxXSxbMSwyXSxbMiwzXSxbNCw1XSxbNSw2XSxbNiw3XSxbNSwxLCJcXGFscGhhX25ee1xcd2lkZXRpbGRle0t9X1xcaW5mdHl9Il0sWzYsMiwie1xcYmV0YX1fbl57XFx3aWRldGlsZGV7S31fXFxpbmZ0eX0iXSxbNywzLCJ7XFxnYW1tYX1fbl57XFx3aWRldGlsZGV7S31fXFxpbmZ0eX09XFx1bmRlcnNldHt2fVxcb3BsdXMgXFwge1xcZ2FtbWF9X3tuLHZ9XntcXHdpZGV0aWxkZXtLfV9cXGluZnR5fSJdXQ==
\[\begin{tikzcd}
	0 & {R^{S}(E/\widetilde{K}_\infty)^{\Gamma_n}} & {H^1_{\fla}(\widetilde{K}_\infty, E[p^\infty])^{\Gamma_n}} & {\left(\underset{v\mid {S}}\prod\ K^1_v(E/\widetilde{K}_\infty)\ \underset{v\nmid {S}}\prod J^1_v(E/\widetilde{K}_\infty)\right)^{\Gamma_n}} \\
	0 & {R^{S}(E/\widetilde{K}_n)} & {H^1_{\fla}(\widetilde{K}_n, E[p^\infty])} & {\underset{v\mid {S}}\prod\ K^1_v(E/\widetilde{K}_n)\ \underset{v\nmid {S}}\prod\ J^1_v(E/\widetilde{K}_n)}
	\arrow[from=1-1, to=1-2]
	\arrow[from=1-2, to=1-3]
	\arrow[from=1-3, to=1-4]
	\arrow[from=2-1, to=2-2]
	\arrow["{\delta_n^{\widetilde{K}_\infty}}", from=2-2, to=1-2]
	\arrow[from=2-2, to=2-3]
	\arrow["{{\beta}_n^{\widetilde{K}_\infty}}", from=2-3, to=1-3]
	\arrow[from=2-3, to=2-4]
	\arrow["{{\gamma}_n^{\widetilde{K}_\infty}=\underset{v}\prod \ {\gamma}_{n,v}^{\widetilde{K}_\infty}}", from=2-4, to=1-4]
\end{tikzcd}\]
Here $\Gamma_n=\operatorname{Gal}(\widetilde K_\infty/
\widetilde K_n)$. Arguing as in \cite[Theorem~4.5]{KleineCanad},
it is enough to establish uniform bounds on the kernel and
cokernel of $\delta_n^{\widetilde K_\infty}$, independently
of $n$ and $\widetilde K_\infty\in\mathcal U(K_\infty,m)$,
for some appropriate choice of $m> 0$. Note that $\operatorname{coker}(\beta_n^{\widetilde K_\infty})=0$,
since $\Gamma_n$ has $p$-cohomological dimension 1.
Now, by applying the snake lemma, this reduces to showing that, for every $\widetilde{K}_\infty \in \calU (K_\infty,m)$, the orders of 
$\ker(\beta_n^{\widetilde K_\infty})$ and
$\ker(\gamma_n^{\widetilde K_\infty})$ are uniformly bounded. We make the following observations:
\begin{enumerate}
    \item $E(K_\infty)[p^\infty]$ is finite (by assumption).

    \item Suppose that $v\in S$ splits completely in
    $K_\infty$. Then $v$ is unramified in every
    $\widetilde K_\infty\in\mathcal U(K_\infty,m)$, since
    $\calP(\widetilde K_\infty)\subseteq \calP(K_\infty)$.
    %Every finite local extension arising from $\widetilde K_\infty/K$ at $v$ is therefore contained in $K_v^{\mathrm{ur}}$.
    By the additional hypothesis and
    Proposition~\ref{prop:split}(i),
    $E(K_v^{\mathrm{ur}})[p^\infty]=0$.
    Hence
    $\ker(\gamma_{n,v}^{\widetilde K_\infty})=0$
    for every $n$.

    \item For $v\nmid S$, we have
    $\ker(\gamma_{n,v}^{\widetilde K_\infty})=0$
    by \cite[Lemma~3.3.1]{Tan}.
\end{enumerate}
Now, using arguments analogous to \cite[Theorem~4.5]{KleineCanad} (see also \cite[Theorem 5.13]{KleineCanad}), we obtain that for $\widetilde K_\infty\in\mathcal U(K_\infty,v_p(|E(K_\infty)[p^\infty]|)+1)$, $\ker(\beta_n^{\widetilde K_\infty})$ is bounded by $|E(K_\infty)[p^\infty]|$ (see \cite[pp.~1404--1405]{KleineCanad}). Similarly, let $v\in S$ not split completely in $K_\infty$, let $w\mid v$ be a prime of $K_\infty$, and let $g_v=p^{e_v}$ be the number of primes of $K_\infty$ above $v$. If $\widetilde K_\infty\in\mathcal U(K_\infty,e_v+v_p(|E(K_{\infty,w})[p^\infty]|)+1)$, then $v$ has exactly $g_v$ primes above it in $\widetilde K_\infty$, and we obtain the uniform bound $|\ker(\gamma_{n,v}^{\widetilde K_\infty})|\le|E(K_{\infty,w})[p^\infty]|^{g_v}$.

For each $v\in S$ that does not split completely in $K_\infty$, fix a prime $w_v\mid v$ in $K_\infty$ and let $\mathrm{loc}_v:=|E(K_{\infty,w_v})[p^\infty]|$. Hence choosing $\widetilde K_\infty\in\mathcal U(K_\infty,m)$ with 
$$m>\max\{v_p(|E(K_\infty)[p^\infty]|),\ e_v+v_p(\mathrm{loc}_v)\},$$ where $v$ runs over the primes of $S$ that do not split completely in $K_\infty$, we obtain a uniform bound on $\ker(\beta_n^{\widetilde K_\infty})$ and $\ker(\gamma_n^{\widetilde K_\infty})$. This completes our proof.
\end{proof}

Now,  mimicking  the ideas of \cite[Theorem 4.11]{KleineCanad}, we obtain the following results using Proposition \ref{prop:FUkudaS1}:
\begin{theorem}\label{mainthm:1}
Let $K_\infty$ be a finitely ramified $\Z_p$-extension of $K$, let $E$ be an
elliptic curve over $K$, and let $\Sigma$ be a finite set of primes of $K$
containing $S$. Assume that
\begin{enumerate}
\item[\textup{(i)}] $E(K_{\infty,w})[p^\infty]$ is finite for every prime $w$ of
$K_\infty$ above $S$;
\item[\textup{(ii)}] $E(K_v)[p]=0$ for every $v\in S$ that splits completely in
$K_\infty$;
\item[\textup{(iii)}] $R^S(E/K_\infty)^\vee$ is a finitely generated torsion
$\Lambda$-module.
\end{enumerate}
Then there exists a neighbourhood $U=\calU(K_\infty,m)$ of $K_\infty$ such that,
for every $\widetilde K_\infty\in U$,
\begin{itemize}
\item $R^\Sigma(E/\widetilde K_\infty)^\vee$ is a torsion $\Lambda$-module;
\item $\mu(R^\Sigma(E/\widetilde K_\infty)^\vee)\le\mu(R^S(E/K_\infty)^\vee)$;
\item if $\mu(R^S(E/\widetilde K_\infty)^\vee)=\mu(R^S(E/K_\infty)^\vee)$, then
$\lambda(R^\Sigma(E/\widetilde K_\infty)^\vee)\le\lambda(R^S(E/K_\infty)^\vee)$.
\end{itemize}
\end{theorem}
\begin{proof}
    The proof follows for $R^{S}(E/\widetilde{K}_\infty)$ by arguments similar to \cite[Theorem 4.11]{KleineCanad}, by application of Proposition~\ref{prop:FUkudaS1}. We briefly recall the key steps. Let $m_0$, $C_1$ and $C_2$ be as in Proposition~\ref{prop:FUkudaS1}. Write $X=R^S(E/K_\infty)^\vee$, $X_i=R^S(E/K_i)^\vee$, $\widetilde X=R^S(E/\widetilde K_\infty)^\vee$ and $\widetilde X_i=R^S(E/\widetilde K_i)^\vee$; since $K_\infty\in\mathcal U(K_\infty,m_0)$, $X$ is a Fukuda module with parameters $(C_1,C_2,1)$. Let $F(X)$ be the maximal finite submodule of $X$ and put $C=v_p(|F(X)|)+v_p(C_1^2C_2)$. Choose $k$ with $k>C$, $p^k>k\lambda(X)+C$ and $p^{k-1}(p-1)>\lambda(X)$, such that $\nu=\nu_{2k,k}=\bigl((1+T)^{p^{2k}}-1\bigr)/\bigl((1+T)^{p^k}-1\bigr)$ is coprime to the characteristic power series of $X$; then $X/\nu X$ is finite. For a $\Lambda$-module $M$ with $M/\nu M$ finite, put $\mathrm{rank}_\nu(M)=v_p(|M/\nu M|)$. Since $|\mathrm{coker}(X\to X_i)|\le C_1$, we have $\mathrm{rank}_\nu(X_i)\le\mathrm{rank}_\nu(X)+v_p(C_1)$ for all $i$. So, for $m\ge m_0$ large, some value is attained by $\mathrm{rank}_\nu(X_i)$ for at least $v_p(C_1^2C_2)+2$ indices $i\le m$. If $\widetilde K_\infty\in\mathcal U(K_\infty,m)$, then $\widetilde K_i=K_i$, hence $\widetilde X_i=X_i$, for $i\le m$, and $\widetilde X$ is a Fukuda module with the same parameters $(C_1,C_2,1)$ by Proposition~\ref{prop:FUkudaS1}. Therefore \cite[Corollary~3.8]{KleineCanad} gives $\mathrm{rank}_\nu(\widetilde X)\le\mathrm{rank}_\nu(X)+v_p(C_1^2C_2)<\infty$, and in particular $\widetilde X$ is torsion. Comparing $\mathrm{rank}_\nu$ of $X$, $\widetilde X$ and their elementary modules \cite[Proposition~3.4]{Kleine}, \cite[pp.~1408--1409]{KleineCanad}, one obtains $p^k(p^k-1)\mu(\widetilde X)\le p^k(p^k-1)\mu(X)+k\lambda(X)+C$. By the choice of $k$, this gives $\mu(\widetilde X)\le\mu(X)$. If equality holds, then $k\lambda(\widetilde X)\le k\lambda(X)+C$, so $\lambda(\widetilde X)\le\lambda(X)$. Now, the result follows from the fact that $R^{\Sigma}(E/\widetilde{K}_\infty)\subseteq R^{S}(E/\widetilde{K}_\infty)$.
\end{proof}

\begin{remark}
 Assume (i) and (ii) of Theorem~\ref{mainthm:1}.  If an analogue of  \nameref{conja} holds for $K_\infty$, meaning that $R^S(E/K_\infty)^\vee$ is a finitely generated $\mathbb{Z}_p$-module, then it follows that an analogue of  \nameref{conja} also holds for $R^\Sigma(E/\widetilde{K}_\infty)$ for every $\widetilde{K}_\infty \in \calU(K_\infty,m)$ and for every $\Sigma$ containing $S$.

\end{remark}
\noindent{We recall the following example from \cite[Remark~3.16(1)]{GhoshJhaShekhar}:}
\begin{example}\label{example}
Let $p=2$, $K=\mathbb{F}_2(t)$, let $K_\infty/K$ be a $\mathbb{Z}_p$-extension obtained from the
Carlitz module, ramified only at $(t)$, and let $E/K$ be
\begin{equation}
  y^2+xy=x^3+(1/t)x^2+1.
\end{equation}
By \cite[Remark 3.16(1)]{GhoshJhaShekhar}, $E$ has bad reduction only at $(t)$; let $\Sigma\supseteq S=\{(t)\}$
be finite. Over $K(z)$, $z^2+z=1/t$, the substitution $y\mapsto y+zx$ identifies $E$ with the
constant curve $E_0\colon y^2+xy=x^3+1$. As $(t)$ is totally ramified in $K_\infty$, this gives
$E(K_{\infty,w})[p^\infty]\subseteq E_0(\mathbb{F}_4)[p^\infty]$ for $w\mid(t)$, so hypothesis (i)
of Theorem~\ref{mainthm:1} holds; (ii) is vacuous; and (iii) holds with
$\mu(R^S(E/K_\infty)^\vee)=0$, as $R^S(E/K_\infty)^\vee$ is finitely generated over $\mathbb{Z}_p$
\cite[Corollary 3.15]{GhoshJhaShekhar}. Using Theorem~\ref{mainthm:1}, we conclude that there is a
neighbourhood $U=\calU(K_\infty,m)$ of $K_\infty$ such that, for every $\widetilde K_\infty\in U$,
$ 
  \mu\bigl(R^\Sigma(E/\widetilde K_\infty)^\vee\bigr)=0
  \quad\text{and}\quad
  \lambda\bigl(R^\Sigma(E/\widetilde K_\infty)^\vee\bigr)\le\lambda\bigl(R^S(E/K_\infty)^\vee\bigr)
$.

Consequently, an analogue of \nameref{conja} of Coates and Sujatha holds for
$R^\Sigma(E/\widetilde K_\infty)^\vee$ for all $\widetilde K_\infty\in U$.
\end{example}

\noindent Recall from \S\ref{subsec:selfinesel} that $K^{\mathrm{sep}}$ and $K_v^{\mathrm{sep}}$ denote separable closures of $K$ and $K_v$, and $G_K$ and $G_{K_v}$ their absolute Galois groups.  In the following proposition, we give several instances when the hypothesis $E(K_v)[p]=0$ is satisfied for every prime $v\in S$ that splits completely  in $K_\infty$.

\begin{proposition}\label{prop:split}
Let $v$ be a prime of $K$ and $K_v^{\mathrm{ur}}$ be the unramified
$\mathbb{Z}_p$-extension of $K_v$.
\begin{enumerate}[(i)]
\item $E(K_v^{\mathrm{ur}})[p^\infty]=0$ if and only if $E(K_v)[p]=0$.
\item If $E(K^{\mathrm{sep}})[p]=0$, $R^S(E/L)=0$ for every $\mathbb{Z}_p$-extension $L/K$ unramified outside $S$.

\item Suppose that $E(K^{\mathrm{sep}})[p]\neq 0$, and let $K(E[p])$ be the field
generated over $K$ by $E(K^{\mathrm{sep}})[p]$. Then $K(E[p])/K$ is cyclic of
degree dividing $p-1$ (so $K(E[p])=K$ if $p=2$), and $E(K_v)[p]\neq 0$ if and only
if $v$ splits completely in $K(E[p])$.
\item Suppose that $p$ is odd and that $E$ has bad reduction at $v$.
\begin{enumerate}[(a)]
\item If $E(K^{\mathrm{sep}})[p]\neq 0$ and the reduction is multiplicative or
potentially multiplicative, then $E(K_v)[p]\neq 0$ if and only if the reduction is
split multiplicative.
\item If the reduction is additive and potentially good ordinary, then
$E(K_v)[p]=0$.
\end{enumerate}
\end{enumerate}
\end{proposition}

\begin{proof}
(i) Clearly $E(K_v)[p]\subseteq E(K_v^{\mathrm{ur}})[p^\infty]$. Conversely, if
$0\neq x\in E(K_v^{\mathrm{ur}})[p]$, its orbit under the pro-$p$ group
$\mathrm{Gal}(K_v^{\mathrm{ur}}/K_v)$ spans a finite non-zero $\mathbb{F}_p$-vector
space on which this group acts through a finite $p$-group, so it has a non-zero
fixed vector, which lies in $E(K_v)[p]$.

(ii) Let $K_\infty^{\mathrm{ur}}$ be the unramified $\mathbb{Z}_p$-extension of $K$
and $L'=LK_\infty^{\mathrm{ur}}$. Then $R^S(E/L')=0$ by
\cite[Remark~7.3(1)]{GhoshRaySuzuki}; the \'etaleness hypothesis there is automatic
as $E(K^{\mathrm{sep}})[p]=0$. The kernel of $R^S(E/L)\to R^S(E/L')$ lies in
$H^1(\mathrm{Gal}(L'/L),E(L')[p^\infty])=0$. Hence $R^S(E/L)=0$.

(iii) Since $E(K^{\mathrm{sep}})[p]\neq 0$, $E$ is ordinary. Fixing
$K^{\mathrm{sep}}\hookrightarrow K_v^{\mathrm{sep}}$, we get
$E(K^{\mathrm{sep}})[p]=E(K_v^{\mathrm{sep}})[p]\cong\mathbb{Z}/p\mathbb{Z}$, as both
lie in $E(\overline{K_v})[p]\cong\mathbb{Z}/p\mathbb{Z}$ \cite[Corollary~III.6.4(c)]{Sil09}.
So $G_K$ acts on it through a character $\psi\colon G_K\to\mathbb{F}_p^\times$ cutting
out $K(E[p])$, which gives the first claim. Since a non-trivial element of
$\mathbb{F}_p^\times$ fixes no non-zero element of $\mathbb{Z}/p\mathbb{Z}$, we have
$E(K_v)[p]\neq0$ if and only if $\psi(G_{K_v})=1$, that is, if and only if $v$ splits
completely in $K(E[p])$.

(iv)(a) By (iii), $E(K_v^{\mathrm{sep}})[p]$ has order $p$. Since $|j(E)|_v>1$
\cite[Proposition VII.5.1(b), Proposition VII.5.5]{Sil09}, $E$ is the twist of a Tate curve $E_q$ by a
character $\chi\colon G_{K_v}\to\{\pm1\}$, trivial exactly in the split
multiplicative case, and
$E_q(K_v^{\mathrm{sep}})\cong(K_v^{\mathrm{sep}})^\times/q^{\mathbb{Z}}$ as
$G_{K_v}$-modules \cite[Lemma~V.5.2, Theorems~V.3.1, V.5.3,
Exercise~V.5.12(a)]{Sil94}. A point of  $E_q(K_v^{\mathrm{sep}})[p]$ is represented by some
$x\in(K_v^{\mathrm{sep}})^\times$ with $x^p\in K_v$. Since $x$ is separable over
$K_v$, we get $x\in K_v$. Hence $G_{K_v}$ acts trivially on $E_q(K_v^{\mathrm{sep}})[p]$
and through $\chi$ on $E(K_v^{\mathrm{sep}})[p]$. As $p$ is odd, $E(K_v)[p]\neq0$ if
and only if $\chi=1$.

(b) Fix a prime $\ell\ge3$ with $\ell\neq p$ and put $M=K_v(E[\ell])$. Let $\mathcal O_M$ be the ring of integers of $M$, $\mathfrak m_M$ its maximal ideal and $k_M=\mathcal O_M/\mathfrak m_M$ its residue field. Then $E$ has good reduction over $M$ \cite[\S2, Corollary 2(b)]{ST68}, and the reduction $\widetilde E_M$ is ordinary, since $E$ has potentially good ordinary reduction. The inertia group $\Phi$ of $M/K_v$ acts on $\widetilde E_M$ by automorphisms, compatibly with reduction \cite[\S2]{ST68}. This action is faithful: $\Phi$ acts faithfully on $E[\ell]$, as $M=K_v(E[\ell])$, and reduction is injective on $E(M)[\ell]=E[\ell]$ \cite[Proposition~VII.3.1(b)]{Sil09}. Moreover, $\Phi\neq1$: otherwise $M/K_v$ would be unramified, so $T_\ell(E)$, being unramified over $M$, would be unramified over $K_v$, and $E$ would have good reduction at $v$ by \cite[Theorem~VII.7.1]{Sil09}.

The reduction map is also injective on $E(M)[p]$. Indeed, let $\widehat E$ be the formal group associated with a minimal Weierstrass equation of $E$ over $\mathcal O_M$; then the kernel of reduction on $E(M)[p]$ is $\widehat E(\mathfrak m_M)[p]$ \cite[Proposition~VII.2.2]{Sil09}. Since $\mathcal O_M$ has characteristic $p$, we have $[p](T)=g(T^{p^h})$ for some power series $g$ with $g'(0)\neq0$ \cite[Proposition~IV.7.2]{Sil09}, where $h\ge1$ because $[p]'(0)=p=0$. As $\widetilde E_M$ is ordinary, its formal group has height one \cite[Theorem V.3.1(b)]{Sil09}, so the reduction of $[p](T)$ is not a power series in $T^{p^2}$. Hence $h=1$ and $g'(0)\in\mathcal O_M^\times$, so $g(y)=0$ with $y\in\mathfrak m_M$ forces $y=0$. Thus $[p](x)=0$ with $x\in\mathfrak m_M$ forces $x^p=0$, that is, $x=0$.

Now let $1\neq\alpha\in\Phi$. Since $\Phi$ is finite, $\alpha$ has finite order, so it acts on $T_p(\widetilde E_M)\cong\mathbb Z_p$ \cite[Corollary~III.6.4(c)]{Sil09} by a root of unity $\zeta$. As the action of $\Phi$ on $\widetilde E_M$ is faithful, $\alpha-1$ is a non-zero isogeny of $\widetilde E_M$, so its kernel is finite \cite[Corollary~III.4.9]{Sil09}. Hence $\zeta\neq1$: otherwise $\alpha-1$ would vanish on $T_p(\widetilde E_M)$, and hence on the infinite group $\widetilde E_M[p^\infty]$. As $p$ is odd, $1+p\mathbb Z_p$ is torsion-free, so $\zeta\not\equiv1\pmod p$. Thus $\alpha$ acts on $\widetilde E_M[p]\cong\mathbb Z/p\mathbb Z$ by an element $\neq1$ of $\mathbb F_p^\times$, which has no non-zero fixed point, and $\widetilde E_M[p]^\Phi=0$. Since reduction is injective on $E(M)[p]$ and $\Phi$-equivariant, $E(K_v)[p]\subseteq E(M)[p]^\Phi=0$.

\end{proof}

\begin{remark}\label{rem:split}
\begin{enumerate}
\item By Proposition~\ref{prop:split}(i), the additional assumption of
Proposition~\ref{prop:FUkudaS1} is equivalent to $E(K_v^{\mathrm{ur}})[p^\infty]=0$
for every $v\in S$ that splits completely in $K_\infty$.
\item If $E(K^{\mathrm{sep}})[p]=0$, then Proposition~\ref{prop:FUkudaS1} and Theorem~\ref{mainthm:1} hold trivially by Proposition~\ref{prop:split}(ii), without the additional assumption. This is the case, for example, if $E$ is supersingular, since then $E[p]=0$ \cite[Theorem~V.3.1(a)]{Sil09}, or if $j(E)\notin K^p$. To see the latter, we show more generally that if $E(K^{\mathrm{sep}})$ contains a point $P$ of order $p^n$ for some $n\geq1$, then $j(E)\in K^{p^n}$; the case $n=1$ gives the claim. Since $E[p]\neq0$, $E$ is ordinary, so $E[p^n]\cong\mathbb Z/p^n\mathbb Z$ \cite[Theorem~V.3.1]{Sil09}, and hence $E[p^n]=\langle P\rangle\subseteq E(K^{\mathrm{sep}})$. As the points of $E[p^n]$ are defined over $K^{\mathrm{sep}}$ and permuted by $G_K$, V\'elu's formulas \cite[(8), (11) and Remark~3]{Velu} give an elliptic curve $E'$ and a separable isogeny $\phi\colon E\to E'$ with kernel $E[p^n]$, both defined over $K$ (cf.\ \cite[Proposition~III.4.12, Remark~III.4.13.3]{Sil09}). By \cite[Corollary~III.4.11]{Sil09}, $[p^n]=\psi\circ\phi$ for an isogeny $\psi\colon E'\to E$. Since $\deg[p^n]=p^{2n}$ \cite[Corollary~III.6.4(a)]{Sil09} and $\deg_s[p^n]=|E[p^n]|=p^n=\deg\phi$ \cite[Theorem~III.4.10(a)]{Sil09}, the isogeny $\psi$ is purely inseparable of degree $p^n$. Hence $\psi$ is the $p^n$-power Frobenius map $E'\to E'^{(p^n)}$ followed by an isomorphism \cite[Corollaries~II.2.12 and~II.2.4.1]{Sil09}, and so $j(E)=j(E')^{p^n}\in K^{p^n}$.

\item If $E(K^{\mathrm{sep}})[p]\neq0$, then by Proposition~\ref{prop:split}(iii) the additional assumption holds if and only if no prime of $S$ splits completely in $K_\infty K(E[p])$, since a prime splits completely in a compositum if and only if it does so in each factor. If $p=2$, then $K(E[p])=K$, so $E(K_v)[p]\neq0$ for every prime $v$ of $K$. Hence, for $p=2$, we have to assume in Proposition~\ref{prop:FUkudaS1} and Theorem~\ref{mainthm:1} that no prime of $S$ splits completely in $K_\infty$. This holds, for instance, if every prime of $S$ ramifies in $K_\infty$, as in Example~\ref{example}.

\item Let $p$ be odd, $E$ isotrivial and $S$ the set of primes of bad reduction and
primes ramified in $K_\infty$. Then the additional assumption holds for every
$K_\infty$. Indeed, if $E$ is supersingular, then $E(K_v)[p]=0$ for all $v$
\cite[Theorem V.3.1]{Sil09}. Otherwise $E$ is ordinary, and since $j(E)$ is constant it has
potentially good ordinary reduction everywhere \cite[Proposition III.1.4(b), Theorem V.3.1,
Proposition VII.5.5]{Sil09}. A prime of $S$ that splits completely in $K_\infty$ is unramified,
hence of bad reduction, and Proposition~\ref{prop:split}(iv)(b) applies.
\end{enumerate}
\end{remark}
\begin{remark}\label{rem:refined}
One can also drop the assumption that $E(K_v)[p]=0$ at the primes $v\in S$ that split completely in $K_\infty$ by passing to smaller neighbourhoods. For $m\geq0$, let $\mathcal U_S(K_\infty,m)$ be the set of those $\widetilde K_\infty\in\mathcal U(K_\infty,m)$ in which every prime of $S$ that splits completely in $K_\infty$ also splits completely. Note that these neighbourhoods depend on the set $S$; they depend on $E$ only through the condition that $S$ contains the primes of bad reduction of $E$. When $S$ consists of the primes of bad reduction and the primes ramified in $K_\infty$, they coincide with the neighbourhoods $\mathcal U(E,K_\infty,m)$ of \cite[\S3]{GR}. For $\widetilde K_\infty\in\mathcal U_S(K_\infty,m)$, the local kernels at the primes of $S$ that split completely in $K_\infty$ vanish, since these primes also split completely in $\widetilde K_\infty$; the rest of the proof of Proposition~\ref{prop:FUkudaS1} is unchanged. Arguing as in \cite[Theorem~3.4]{GR}, Proposition~\ref{prop:FUkudaS1} and Theorem~\ref{mainthm:1} then hold without this assumption, with $\mathcal U(K_\infty,m)$ replaced by $\mathcal U_S(K_\infty,m)$. We omit the details.
\end{remark}

\section{Asymptotic growths over $\Z_p^2$-extensions}\label{sec:asymptoticgrowths}

In this section, we study the asymptotic growth of Iwasawa invariants in a $\Z_p^2$-extension of $K$. 
For this, we adopt the techniques established in \cite{KleineMatarSujatha}.
\subsection{Setup}
Let $K$ be  a global  function field of characteristic $p$ over a finite field $\Ff$ as above. Let $K_\infty$ and $K_\infty'$ be two finitely ramified $\Z_p$-extensions of $K$, such that $K_\infty\cap K_\infty'=K$. Set $\calL:=K_\infty K_\infty'$ and $G:=\gal(\calL/K)\cong \Z_p^2$.
Let $K_n$ be the unique subfield of $K_\infty$ of degree $p^n$ over $K$. Put $H:=\gal(\calL/K_\infty')$ and $H_n:=H^{p^n}\cong p^n\Z_p$.
We also consider the series of $\Z_p$-extensions $F_{\infty,n}$ of $K_n$, where $F_{\infty,n}=K_nK_\infty'$. Note that $F_{\infty,n}=\calL^{H_n}$. We put all this data in the following diagram:
\begin{equation}\label{diag:fields}
\scalebox{0.85}{%
\begin{tikzcd}[ampersand replacement=\&]
    \& \& {\calL} \\
    \& {F_{\infty,n}} \& \& {K_\infty} \\
    {K_\infty'} \& \& {K_n} \\
    \& K
    \arrow[no head, from=1-3, to=2-4]
    \arrow[no head, from=3-3, to=2-2]
    \arrow[no head, from=2-2, to=1-3]
    \arrow["{H_n}", bend right=45, no head,
           from=2-2, to=1-3]
    \arrow["H", bend left=35, no head,
           from=3-1, to=1-3]
    \arrow[no head, from=3-1, to=2-2]
    \arrow[no head, from=3-1, to=4-2]
    \arrow[no head, from=3-3, to=2-4]
    \arrow[no head, from=4-2, to=3-3]
\end{tikzcd}%
}
\end{equation}
\iffalse
\begin{equation}\label{diag:fields}
    \begin{tikzcd}
	&& {\calL} \\
	& {F_{\infty,n}} && {K_\infty} \\
	{K_\infty'} && {K_n} \\
	& K
	\arrow[no head, from=1-3, to=2-4]
        \arrow[no head, from= 3-3, to=2-2]
	\arrow[no head, from=2-2, to=1-3]
	\arrow["{H_n}", bend right=45, no head, from=2-2, to=1-3]
	\arrow["H", bend left=35, no head, from=3-1, to=1-3]
	\arrow[no head, from=3-1, to=2-2]
	\arrow[no head, from=3-1, to=4-2]
	\arrow[no head, from=3-3, to=2-4]
	\arrow[no head, from=4-2, to=3-3]
\end{tikzcd}
\end{equation}
\fi
Choose two topological  generators $\tilde{\sigma}$ and $\tilde{\tau}$ of $G\cong \Z_p^2$  such that $H :=\gal(\calL/K_\infty')$ is  generated topologically by $\tilde{\sigma}$ and $\gal(\calL/K_\infty )$  is generated topologically by $\tilde{\tau}$. Then
$\Lambda(G)$  can be identified with the ring of formal power series $\Z_p[[T, U ]]$  in two variables, where $U = \tilde{\sigma}-1$ and $T =\tilde{\tau}-1$.

For any $n\in \N$, let $\eta_n=(U+1)^{p^n}-1$ and define $\alpha_{n,m}=\frac{\eta_n} {\eta_m}$ {, where $n\geq m$}. For fixed $m$, we use the shorthand $\alpha_n$ instead of $\alpha_{n,m}$. 
We will recall the following lemmas from \cite{cuoco}.
\begin{lemma}\cite[Lemma 2.8] {cuoco}\label{lemma:cuoco1}
    Let $V$ be a finitely generated torsion $\Lambda(G)$-module, and let $N$ be a pseudo-null submodule. Suppose that $(V_n)_{n\in \N}$ is a family of submodules of $V$ and that there
exists an integer $n_0\in \N$ such that the following hold for each integer $n>n_0$, $V_n=\alpha_{n,n_0}V_{n_0},$ and $\eta_nV=V_n$.

Then, $(N+V_n)/V_n$ is a finitely generated and torsion $\Lambda(G/H)$-module for each $n>n_0$, and the invariants $\mu_{G/H}((N+V_n)/V_n)$ and $\lambda_{G/H}((N+V_n)/V_n)$ become constant for sufficiently large $n$.
\end{lemma}
\begin{lemma}(\cite[Proposition 2.1]{cuoco}, \cite[Lemma 6.6]{KleineMatarSujatha})\label{lemma:cuoco2}
    Let $W$ be a finitely generated and torsion $\Lambda(G)$-module, and suppose that there exists an integer $n_0$ such that the quotient $W/\alpha_n W$ is a finitely generated and torsion $\Lambda(G/H)$-module for each $n > n_0$. Then there exist constants $l, m_0 , m_1 , c_1, c_2$, independent of $n$, such that
    \[
    \mu_{G/H}(W/\alpha_n W)=m_0p^n+m_1n+c_1,\ \lambda_{G/H}(W/\alpha_n W)=lp^n+c_2
 \]
 for sufficiently large $n\in \N$. Here $m_0 = \mu_G(W)$ denotes the largest power of $p$ which divides the characteristic power series of $W$ in $\Lambda(G)$. 
\end{lemma}
The following hypothesis is important for our proofs in this section.
\vspace{0.75cm}
\mystatement{\textbf{\underline{Hyp Tors}} \label{Hyp Tors}: We assume that $R^S(E/\calL)^\vee$ and  $R^S(E/F_{\infty,n})^\vee$ are finitely generated torsion $\Lambda(G)$- and $\Z_p[[\Gal(F_{\infty,n}/K_n)]]$-modules respectively for all $n\geq 0$.}{stmt:first}

\begin{remark}\label{rem:hyptors}
If $E$ is non-isotrivial, then \myref{stmt:first}{Hyp Tors} holds for every $\calL$, defined in the setup. Indeed, let $e\geq0$ be the largest integer with $j(E)\in K^{p^e}$. It exists because $j(E)\notin\mathbb F$, whereas an element of $K$ lying in $K^{p^n}$ for every $n$ has all its valuations divisible by every power of $p$, hence has no poles and lies in $\mathbb F$. By the argument of Remark~\ref{rem:split}(2), $E(K^{\mathrm{sep}})$ has no point of order $p^{e+1}$. Since $E(K^{\mathrm{sep}})[p^\infty]$ is a subgroup of $E[p^\infty]$, which is $0$ or isomorphic to $\mathbb Q_p/\mathbb Z_p$ \cite[Corollary~III.6.4(c)]{Sil09}, it is therefore killed by $p^e$.
Put $\calL'=\mathcal{L}K_\infty^{\mathrm{ur}}$, where $K_\infty^{\mathrm{ur}}$ is the unramified $\mathbb{Z}_p$-extension of $K$. Then $R^S(E/\calL')$ is killed by $p^e$: apply \cite[Proposition~4.4]{GhoshRaySuzuki} to the finite layers of $\calL'\overline{\mathbb{F}}/K\overline{\mathbb{F}}$ and use that $\mathrm{Gal}(\calL'\overline{\mathbb{F}}/\calL')$ is pro-prime-to-$p$ (cf.\ the proof of \cite[Theorem~7.1(3)]{GhoshRaySuzuki} and \cite[ Proposition~5.1]{GhoshRaySuzuki}). By inflation--restriction, the kernel of $R^S(E/\mathcal{L})\to R^S(E/\calL')$ lies in $H^1(\mathrm{Gal}(\calL'/\mathcal{L}),E(\calL')[p^\infty])$, which is killed by $p^e$. So $R^S(E/\mathcal{L})$ is killed by $p^{2e}$. It is cofinitely generated over $\Lambda(G)$ by Nakayama's lemma: $R^S(E/K)\subseteq\mathrm{Sel}(E/K)$ is cofinitely generated over $\mathbb{Z}_p$, and the cokernel of $R^S(E/K)\to R^S(E/\mathcal{L})^G$ is finite by the snake lemma, as in the proof of Proposition \ref{propn:controlHn}, using Lemmas~\ref{lemma:BandiniLonghiIJNT} and \ref{lemma:finitenesslocal} and the finiteness of $H^2(G,E(\mathcal{L})[p^\infty])$.
Hence $R^S(E/\calL)^\vee$ is a finitely generated
torsion $\Lambda(G)$-module, and the same argument over $K_n$ applies to
$R^S(E/F_{\infty,n})^\vee$.

If $E$ is isotrivial and $K'_\infty$ is the unramified $\mathbb{Z}_p$-extension of
$K$, as in Remark~\ref{rem:propn3.4}(b), then \myref{stmt:first}{Hyp Tors} holds under the assumptions of
\cite[Theorem~7.2]{GhoshRaySuzuki}. By its proof, $R^S(E/\calL)^\vee$ is finitely generated over
$\Lambda(H)$, so it is $\Lambda(G)$-torsion. Under the hypotheses of Proposition \ref{propn:controlHn}, each
$R^S(E/F_{\infty,n})^\vee$ is then finitely generated over $\mathbb{Z}_p$.
\end{remark}
\subsection{Main results}
By our hypothesis \myref{stmt:first}{Hyp Tors}, $R^S(E/F_{\infty,n})^\vee$ is a finitely generated torsion $\Lambda_n:=\Z_p[[\gal(F_{\infty,n}/K_n)]]$-module. We identify each of these { group rings} with $\Z_p[[T]]$. We now denote the $\mu$- and $\lambda$-invariants of $R^S(E/F_{\infty,n})^\vee$ by $\mu_n$ and $\lambda_n$. 

\begin{proposition}\label{propn:controlHn}
    Let $E$ be an elliptic curve defined over $K$. Let $S$ be a finite set of primes of $K$ that contains the primes ramified in $\calL$ and the primes where $E$ has  bad reduction.
    Furthermore, assume that

    \begin{enumerate}[(i)]
        \item for every prime $w\mid S$ of $\calL$,  $E(\calL_w)[p^\infty]$ is finite. 
        \item every prime $v\mid S$ is finitely split in $\calL$.
    \end{enumerate}
    
    Then the kernel and the cokernel of the map 
    \[
    \phi_n: R^S(E/\calL)^\vee/\eta_n\lra R^S(E/F_{\infty,n})^\vee 
    \]
     { are finite,
and their orders are bounded independently of $n$.  In particular, $R^S(E/\mathcal{L})^\vee/\eta_n$ and $R^S(E/F_{\infty,n})^\vee$ have the same $\mu$- and $\lambda$-invariants over $\mathbb{Z}_p[[\Gal(F_{\infty,n}/K_n)]]$.}
    
    %torsion $\Z_p[[\Gal(F_{\infty,n}/K_n)]]$-modules with $\mu$-invariants bounded independently of $n$ and trivial $\lambda$-invariants.
\end{proposition}
\begin{proof}
    Note that the map $ \phi_n: R^S(E/\calL)^\vee/\eta_n\lra R^S(E/F_{\infty,n})^\vee $ is dual to the map $R^S(E/F_{\infty,n})\lra R^S(E/\calL)^{H_n}$, where $H_n=\gal(\calL/F_{\infty,n})$. Consider the following commutative diagram:
    \[\begin{tikzcd}
	0 & {R^S(E/\calL)^{H_n}} & {H^1_{\fla}(\calL, E[p^\infty])^{H_n}} & {\left(\underset{v\mid S}\prod\ K^1_v(E/\calL)\ \underset{v\nmid S}\prod\ J^1_v(E/\calL)\right)^{H_n}} \\
	0 & {R^S(E/F_{\infty,n})} & {H^1_{\fla}(F_{\infty,n}, E[p^\infty])} & {\underset{v\mid S}\prod\ K^1_v(E/F_{\infty,n})\ \underset{v\nmid S}\prod\ J^1_v(E/F_{\infty,n})}
	\arrow[from=1-1, to=1-2]
	\arrow[from=1-2, to=1-3]
	\arrow[from=1-3, to=1-4]
	\arrow[from=2-1, to=2-2]
	\arrow["\delta_n",from=2-2, to=1-2]
	\arrow[from=2-2, to=2-3]
	\arrow["{\beta_n}", from=2-3, to=1-3]
	\arrow[from=2-3, to=2-4]
	\arrow["{\gamma_n=\underset{v}\prod \gamma_{n,v}}", from=2-4, to=1-4]
\end{tikzcd}\]

Since $H_n\cong\mathbb{Z}_p$ has $p$-cohomological dimension $1$, $\operatorname{coker}(\beta_n)=0$, and $\ker(\beta_n)=H^1(H_n,E(\mathcal{L})[p^\infty])$ has order at most $|E(\mathcal{L})[p^\infty]|$ by Lemma~1.4; this is finite by (i), as $E(\mathcal{L})\subseteq E(\mathcal{L}_w)$. For $v\notin S$, $\ker(\gamma_{n,v})=0$ by \cite[Lemma~3.3.1]{Tan}. For $v\in S$, $\ker(\gamma_{n,v})$ is the product, over the primes of $F_{\infty,n}$ above $v$, of groups $H^1(H_{n,w},E(\mathcal{L}_w)[p^\infty])$ with $H_{n,w}\subseteq H_n$ the decomposition group; each has order at most $|E(\mathcal{L}_w)[p^\infty]|$ by (i) and Lemma~\ref{lemma:BandiniLonghiIJNT}, and by (ii) there are at most $g_v$ factors, $g_v$ being the number of primes of $\mathcal{L}$ above $v$. By the snake lemma, $|\operatorname{coker}(\varphi_n)|=|\ker(\delta_n)|\le|\ker(\beta_n)|$ and $|\ker(\varphi_n)|=|\operatorname{coker}(\delta_n)|\le|\ker(\gamma_n)|$, and these bounds do not depend on $n$.
\end{proof}
%It is easy to deduce from assumption (i)  that both $\ker(\beta_n)$ and $\ker(\gamma_{n,v})$ are finite for $v\mid S$ and for $v\nmid S$ $\ker(\gamma_{n,v})=0$ by \cite[Lemma 3.3.1]{Tan}. Moreover, their orders are bounded for each $n>0$ by  assumption (ii) and Lemma \ref{lemma:BandiniLonghiIJNT}. The desired result then follows immediately from an application of the snake lemma.

\begin{remark}\label{rem:propn3.4}
    \begin{enumerate}[(a)]
        \item If $E/K$ is a non-isotrivial elliptic curve, then assumption (i) of Proposition \ref{propn:controlHn} holds by Lemma \ref{lemma:finitenesslocal}.
        \item Assume $\calL$ is the compositum of the unramified $\Z_p$-extension $K_\infty'$ and a $\Z_p$-extension $K_\infty$ obtained from Carlitz modules (constructed in \S\ref{subsec:examplesofZpextensions}). Here, $K_\infty'$ is everywhere unramified, and { $K_\infty$} is ramified only at one prime, say $v_{\text{ram}}$. Let $S=\{v_{\text{ram}}\}$ and assume $E$ has good reduction outside $S$ (see, for example,  \cite[Remark 3.16]{GhoshJhaShekhar}). Then condition (ii) of Proposition \ref{propn:controlHn} holds in this setting.
    \end{enumerate}
\end{remark}
Put $X:=R^S(E/\mathcal{L})^\vee$. Under \myref{stmt:first}{Hyp Tors}, $X$ is a finitely generated torsion $\Lambda(G)$-module; let $F_X\in\Lambda(G)$ be its characteristic power series, which is well defined up to a unit. We put $\mu_G(X):=m_0(F_X)$ and $l_0(X):=l_0(F_X)$ (Definitions \ref{defn:constantm0} and \ref{defn:constantl0}); these do not depend on the choice of $F_X$. Finally, let $X[p^\infty]$ be the $\Lambda(G)$-submodule of elements of $X$ killed by some power of $p$, and put $X_f:=X/X[p^\infty]$.

\begin{theorem}\label{thm:asymgrowth}
    Assume that the setting and { hypotheses} of Proposition \ref{propn:controlHn} hold. Further assume that \myref{stmt:first}{Hyp Tors} holds. 
    Then the following assertions hold:
    \begin{enumerate}[(a)]
        \item For suitable constants $m_1(H)$ and $l(H)$, we have
        \begin{equation}\label{equn:mun}
            \mu_n=\mu_G(R^S(E/\calL)^{\vee})p^n+m_1(H)n+O(1), 
        \end{equation}
        and
        \begin{equation}\label{equn:lambdan}
            \lambda_n=l(H)p^n+O(1).
        \end{equation}
       
        \item The parameter $m_1(H)$ is zero for all but finitely many choices of $H$, and 
        \[
        \underset{H\subset G}\Sigma m_1(H)=l_0(X),
        \]
        where  $H$ runs over the closed
subgroups of $G=\mathrm{Gal}(\calL/K)$ with $G/H\cong\mathbb{Z}_p$.
        
        \item If $R^S(E/\calL)^\vee_f:=R^S(E/\calL)^\vee/R^S(E/\calL)^\vee[p^\infty]$ is finitely generated as a $\Lambda(H)$-module, then $l(H)=\rank_{\Lambda(H)}(R^S(E/\calL)^\vee_f)$.
    \end{enumerate}
\end{theorem}
\begin{proof}
We identify $\mathbb{Z}_p[[\mathrm{Gal}(F_{\infty,n}/K_n)]]$ with
$\Lambda(G/H)=\mathbb{Z}_p[[T]]$. By Proposition \ref{propn:controlHn}, $X/\eta_nX$ and
$R^S(E/F_{\infty,n})^\vee$ differ by finite modules. So $\mu_n$ and $\lambda_n$ are
the invariants of the $\Lambda(G/H)$-module $X/\eta_nX$, which is finitely generated
and torsion by \myref{stmt:first}{Hyp Tors}. The rest of the proof is module theoretic in nature, as in \cite[Lemma~6.7]{KleineMatarSujatha}. Let
$\varphi\colon X\to M$ be a pseudo-isomorphism to an elementary $\Lambda(G)$-module,
with kernel $N$ and image $R$. Put $Y_n=\eta_n\cdot X$ and $W_n=\varphi(Y_n)$, so
that $W_n=\alpha_{n,0}W_0$. Then $\varphi$ induces surjections
$X/Y_n\twoheadrightarrow R/W_n$ with kernel $(N+Y_n)/Y_n$, whose invariants are
constant for $n\gg0$ by Lemma \ref{lemma:cuoco1}. Moreover
$\mu(R/W_n)=\mu(R/W_0)+\mu(W_0/\alpha_{n,0}W_0)$, and similarly for $\lambda$.
Applying Lemma \ref{lemma:cuoco2} to $W_0$ gives \eqref{equn:mun} and \eqref{equn:lambdan}. Here
$\mu_G(W_0)=\mu_G(X)$, since $W_0$ is pseudo-isomorphic to $Y_0$, and $Y_0$ is
pseudo-isomorphic to $X$ because $X/Y_0$ is a finitely generated torsion
$\Lambda(G/H)$-module. 

For (b), write the characteristic power series of $X$ as $p^{\mu_G(X)}g_X$ with $p\nmid g_X$, and for a closed subgroup $H$  of $G$ with $G/H\cong \Z_p$, put $l_H(X):=\upsilon_{P_H}(\bar g_X)$, where $P_H=(\overline{\sigma_H-1})$ for a topological generator $\sigma_H$ of $H$. By \cite[Theorem~2.4]{Cuoco84} (see the proof of \cite[Lemma~6.9]{KleineMatarSujatha}), $m_1(H)=l_H(X)$ whenever (a) applies to $H$, and we put $m_1(H):=l_H(X)$ for the remaining $H$. As $\bar g_X$ has only finitely many prime factors, $m_1(H)=0$ for all but finitely many $H$, and $\sum_H m_1(H)=l_0(X)$ by Definition~1.7 and \cite[Lemma~1]{Cuoco82}.

%For (b), $m_1(H)$ agrees with the constant in (a) by\cite[Theorem~2.4]{Cuoco84}, and the remaining assertions follow from \cite[Lemma~1]{Cuoco82}; see \cite[Lemma~6.9]{KleineMatarSujatha}. Put $X_f:=X/X[p^\infty]=R^S(E/\mathcal{L})^\vee_f$, where $X[p^\infty]$ is the $\Lambda(G)$-submodule of elements of $X$ killed by some power of $p$. 

For (c), compare \eqref{equn:lambdan}
with the formula
$\mathrm{rank}_{\mathbb{Z}_p}((X_f)_{H_n})=\mathrm{rank}_{\Lambda(H)}(X_f)\,p^n+O(1)$
of Harris \cite{Har00}, as in \cite[Lemma~6.9]{KleineMatarSujatha}.
\end{proof}

\begin{remark}\label{rem:nonisotrivial}
If $E$ is non-isotrivial, Theorem \ref{thm:asymgrowth} simplifies. By Remark~\ref{rem:hyptors},
$R^S(E/\calL)^\vee$ and all $R^S(E/F_{\infty,n})^\vee$ are killed by a fixed power of $p$.
Hence $\lambda_n=0$ for all $n$, so $l(H)=0$ and $R^S(E/\calL)^\vee_f=0$. Moreover, the characteristic power series of $X=R^S(E/\mathcal{L})^\vee$ is a power of $p$, so $\bar g_X=1$ and $m_1(H)=l_H(X)=0$ for every $H$ by \cite[Theorem~2.4]{Cuoco84}; in particular $l_0(X)=0$. Thus, \eqref{equn:mun} reads $\mu_n=\mu_G(R^S(E/\calL)^\vee)\,p^n+O(1)$, and the
$\lambda$-estimates of Theorem \ref{thm:asymgrowth}  are non-trivial only when $E$ is isotrivial.
\end{remark}
\begin{remark}[Isotrivial versus non-isotrivial curves]\label{rem:summary}
We summarise how the results of this article depend on whether $E$ is isotrivial.
\begin{enumerate}
\item Let $E$ be non-isotrivial. Then the groups $E(K_{\infty,w})[p^\infty]$ and $E(\mathcal L_w)[p^\infty]$ are finite (Lemma~\ref{lemma:finitenesslocal}), so hypothesis~(i) of Theorem~\ref{mainthm:1} and of Proposition~\ref{propn:controlHn} holds automatically. Moreover, $E(K^{\mathrm{sep}})[p^\infty]$ is finite, and the fine Selmer groups over the $\mathbb Z_p$- and $\mathbb Z_p^2$-extensions considered here are killed by a bounded power of $p$ (Remark~\ref{rem:hyptors} and its proof). Hence hypothesis~(iii) of Theorem~\ref{mainthm:1} and Hyp~Tors hold automatically, and all $\lambda$-invariants vanish. In Theorem \ref{thm:asymgrowth}, we have $l(H)=m_1(H)=l_0(X)=0$ and $\mu_n=\mu_G(X)p^n+O(1)$ (Remark~\ref{rem:nonisotrivial}). Thus the $\mu$-invariants carry all the information.
\item Let $E$ be isotrivial. If $p$ is odd and $S$ consists of the primes of bad reduction and the primes ramified in $K_\infty$, then hypothesis~(ii) of Theorem~\ref{mainthm:1} holds automatically (Remark~\ref{rem:split}(4)); hypothesis~(i) has to be checked, as in Example~\ref{example}. If moreover $K'_\infty$ is the unramified $\mathbb Z_p$-extension of $K$ and the assumptions of \cite[Theorem~7.2]{GhoshRaySuzuki} and of Proposition~\ref{propn:controlHn} hold, then $X$ is finitely generated over $\Lambda(H)$ and each $R^S(E/F_{\infty,n})^\vee$ is finitely generated over $\mathbb Z_p$ (Remark~\ref{rem:hyptors}). Hence $\mu_n=0$ for all $n$, so $\mu_G(X)=m_1(H)=0$ by Theorem~\ref{thm:asymgrowth}(a), and $\lambda_n=l(H)p^n+O(1)$ with $l(H)=\operatorname{rank}_{\Lambda(H)}(X_f)$ by Theorem~\ref{thm:asymgrowth}(c). Thus the $\lambda$-invariants carry all the information.
\end{enumerate}
\end{remark}

\section*{Acknowledgement}
The  author appreciates the support received from the NBHM postdoctoral fellowship. The author thanks the anonymous referee for the helpful comments and suggestions that have improved this article.

\bibliographystyle{alpha}
\bibliography{bib}
\end{document}